\documentclass{gSTA2e}

\theoremstyle{plain}
\newtheorem{theorem}{Theorem}[section]

\theoremstyle{remark}
\newtheorem{remark}[theorem]{Remark}

\newcommand{\E}{{\rm E}}
\newcommand{\Var}{{\rm Var}}
\newcommand{\net}{\chi^2_{{\rm net}}}
\newcommand{\netnewer}{\chi^2_{{\rm net},1}}
\newcommand{\netold}{\chi^2_{{\rm net},2}}
\newcommand{\tnet}{t^2_{{\rm net}}}

\begin{document}



\title{{\itshape A modified $\chi^2$-test for uplift models with applications in marketing performance measurement}}

\author{Ren{\'e} Michel$^{\ast}$\thanks{$^\ast$Corresponding author. Email: rene.michel@altran.com
\vspace{6pt}}, Igor Schnakenburg and Tobias von Martens\\\vspace{6pt}
{\em{Altran GmbH \& Co. KG,
Schillerstra\ss e 20,
    60313 Frankfurt, Germany}}\\\vspace{6pt}\received{v1.0 sent January 2014} }

\maketitle

\begin{abstract}
Uplift, essentially being the difference between two
probabilities, is a central number in marketing performance
measurement. A frequent question in applications is whether the
uplifts of two campaigns are significantly different. In this article we present
 a new $\chi^2$-statistic which allows
to answer this question by performing a statistical test. We show
that this statistic is asymptotically $\chi^2$-distributed and demonstrate its application in a real life
example.
By running simulations with this new  and alternative approaches, we find our suggested test to exhibit a better decisive power.


\begin{classcode}62F03
 \end{classcode}

\end{abstract}

\section{Introduction}\label{sec:introduction}

Nowadays companies do not wait for customers to contact them and show their interest in a certain product. They rather try to initiate the contact themselves. This can be done via channels
such as mail, email or phone call. The company-initiated, personalised targeting of a fixed group of persons with a specific topic is usually referred to as a
campaign. These campaigns are part of the so called below-the-line marketing since only the targeted group is contacted.
An integral part of the campaign management process is to measure, whether a campaign was successful in order to be able to optimize the company's campaign portfolio.
Since customer behaviour is subject to random fluctuations,
statistical methods are needed to distinguish between ''true'' effects and success by chance.

One can measure campaign effectiveness only on the basis of differences between response rates in a target group and a structurally similar
group that has not been addressed with the campaign (control group). This difference is usually referred to as uplift, see~\cite{radcliffe:simpson:2008}.
Only in case it is positive, a campaign has generated added value and may therefore be regarded as effective, see~\cite{lo:2002}. An overview of the
statistics in campaign performance measurement is given in Chapter~8 of~\cite{falk:2014}.

In this paper we present a statistical method of comparing success of two campaigns, or more precisely, of comparing their uplifts. We do this by introducing $\net$, a modified version
of the classical $\chi^2$-statistic. With this new statistic and by showing that it asymptotically follows a $\chi^2$-distribution, we are able to
discover significant differences between the uplifts of two campaigns. We will also show its application to real data examples.

Up to now, only few studies cover this area of research. They primarily deal with the context of net scoring in which our statistic can also be used.
Net scoring is a generalisation of classical scoring procedures in that it does not predict the probability of a desired customer behaviour,
but instead predicts the increase of this probability through targeting.
Some algorithms used in this context apply statistics in order to compare the uplifts of two campaigns. The most prominent example
is~\cite{radcliffe:surry:2011}. Here, an $F$-statistic is introduced which tackles this problem. Another example is~\cite{michel:schnakenburg:martens:2013} where an ancestor
of our statistic $\net$ is introduced. Both papers primarily show the statistics and apply them as parts of complex algorithms but do not investigate their statistical
background. In this paper, we focus on the mathematical background. We construct a sound statistical testing procedure for the comparison of
campaign success which includes the computation of asymptotical distributions.

The alternatives to our approach are based on variance analysis. The first is the above mentioned $F$-statistic, the second would be the use of contrasts.
We will also introduce these methods and compare them to our new approach, mainly by simulations.

We begin by summarizing the basics of campaign performance measurement, especially control groups, in Section~\ref{sec:control groups}. After that, we present the
test scenarios which we investigate in this paper. Section~\ref{sec:modified chi square}, being the main section of our article, introduces the $\net$-statistic
and proves its asymptotic $\chi^2$-distribution. The alternative approaches are summarized in Section~\ref{sec:alternative}.
They are compared in Section~\ref{sec:simulation} by means of a simulation study.
In Section~\ref{sec:application}, we show an application of the statistics to a real data example.
A discussion and an outlook to yet unsolved problems in Section~\ref{sec:discussion} conclude this paper.

\section{Control groups}\label{sec:control groups}

In order to isolate the impact of campaigns, control groups are necessary.
A control group is a structurally identical group of customers that is not targeted and, hence, reflects customer behaviour without campaign impact.
In other words, the control group has to be representative of the target group. This is usually ensured by a random selection.
Differences in customer behaviour between target and control group may be represented by differences in the respective response rates, see Chapter~8 of~\cite{falk:2014}.

By campaigning, a company usually wishes to induce a certain response behaviour on the targeted people. If a person shows this behaviour within a certain time
period after the contact, this is regarded as a response.
The number of responses in relation to the number of targeted persons is called response rate and is our means of comparing target and control group

$$\mbox{response rate}=\frac{\mbox{number of responses}}{\mbox{number of targeted persons}}$$

 Classical examples of responses are product purchases, the acquisition of a new customer, appointments with sales personnel or the cancellation of business relationships (churn events, in which case
 a low rate is considered ''good'').

The response rate in the group of people that have been targeted does not suffice
to estimate the success of the campaign under consideration. Campaign effectiveness is measured by comparing it to the rate within the control group.
This difference is usually referred to as uplift.

\begin{eqnarray*}\mbox{uplift}=\mbox{response rate within target group}-\mbox{response rate within control group}\end{eqnarray*}

Only in the case it is positive, a campaign has generated added
 value and may therefore be regarded effective. The response rate within the control group is often referred to as random noise.

The uplift multiplied by the size of the target group, describes the additional responses gained by the campaign and thus the net impact of that campaign.
Examples are additional product sales, additional customers won, additional appointments made or prevented churn events.

This method of measuring the campaign effect is subject to random fluctuations. One classical problem is the question of whether the uplift
is different from $0$ only by random fluctuations or if a ''true'' impact is present.
This problem can be solved by Fisher's exact test or the classical $\chi^2$-test for homogenity for the comparison of response rates within target and control group,
see Sections~4.1 and~4.2 of~\cite{falk:2003} and Section~8.3 of~\cite{falk:2014}. Typically, the null hypothesis assumes no impact, i.e. uplift $=0$.

The problem we wish to discuss in this paper, however, is more complicated. We assume that we have two campaigns and we want to compare their success to each other,
i.e. we want to know if their uplifts are significantly different from each other.

\section{Test scenarios}\label{sec:test scenarios}

In describing our test scenarios another rate will play an important role. Since the validity of our measuring method depends on the target and control groups sizes
we define the target-control rate as follows:

$$\mbox{target-control rate}=\frac{\mbox{number of persons in target group}}{\mbox{number of persons in control group}}$$

In Section~\ref{sec:modified chi square}, we will present our test statistic. But first we will explain the
two scenarios to which this statistic can be applied.

\begin{itemize}
\item Scenario 1:\\
Imagine two different campaigns which should be compared regarding their uplift. In order to provide useful results in the marketing context, they should be comparable in some way, e.g.
have been carried out at roughly the same time or have advertised a similar product.

\item Scenario 2:\\
Imagine a campaign with target and control group. Within this campaign there are two different groups of interest,
for example men and women or people below and above the age of $40$. We are interested in whether the uplifts within both subgroups differ significantly from each other.
\end{itemize}

Mathematically, one could argue that Scenario 2 is part of Scenario 1. This is true. Assuming a representative control group of the main campaign, the
splitting criterion should separate target and control group in an even way. Especially the target-control rate in both subgroups should roughly be equal. In Scenario 1, we make no assumption
on the target-control rate. Thus, the case of equal target-control rate in both campaigns is a special case and hence included in Scenario~1. However, we distinguish between these two scenarios, since they
represent two different cases from the perspective of campaign management. In the first case, ''independent'' campaigns are compared, in the latter case
only one campaign is subject to the investigation and is artificially split. Scenario 2 is mainly the way net scoring is done when one uses scoring procedures
like decision trees.

\section{The $\net$-statistic}\label{sec:modified chi square}

In order to introduce our new $\net$-statistic, we firstly formalise the setup.
We start by assuming two campaigns, each containing target and control group either resulting from one campaign, split into two subgroups by some criterion, or two separate campaigns).
In order to ease the notation, we refer to those individual campaigns from now on
as subgroups (still covering both Scenarios~1 and~2). If we unify target and control groups, we assume
that we have $n$ observations in the unified target group and $k$ observations in the control unified group.  We further assume that
subgroup~1 contains $n_1$ target observations and $k_1$
control observations, subgroup~2 $n_2$ and $k_2$ observations respectively. By $a_{G,SG}$, we denote the number of responses in the respective group where $G=T, C$ for target
and control group and $SG=1,2$ or missing for the first, the second subgroup or the overall group. In tabular form, this reads as
\begin{equation} \label{tab:target_group}\begin{array}{l|l|l|l}
\mbox{target}\quad &\quad \mbox{response}\quad &\quad \mbox{no response}\quad &\quad \mbox{total}\quad \\
\hline
\mbox{subgroup~1}\quad &\quad a_{T,1}\quad &\quad n_1-a_{T,1}\quad &\quad n_1\quad\\
\hline
\mbox{subgroup~2}\quad &\quad a_{T,2}\quad &\quad n_2-a_{T,2}\quad &\quad n_2\quad\\
\hline
\mbox{total}\quad &\quad a_{T}\quad &\quad n-a_{T}\quad &\quad n\quad\\
 \end{array}
 \end{equation}

for the target group and for the control group

 \begin{equation} \label{tab:control_group}\begin{array}{l|l|l|l}
\mbox{control}\quad &\quad \mbox{response}\quad &\quad \mbox{no response}\quad &\quad \mbox{total}\quad\\
\hline
\mbox{subgroup~1}\quad &\quad a_{C,1}\quad &\quad k_1-a_{C,1}\quad &\quad k_1\quad\\
\hline
\mbox{subgroup~2}\quad &\quad a_{C,2}\quad &\quad k_2-a_{C,2}\quad &\quad k_2\quad\\
\hline
\mbox{total}\quad &\quad a_{C}\quad &\quad k-a_{C}\quad & \quad k\quad\\
\end{array}
 \end{equation}

We now define the additional responses (or uplift-based responses) of the subgroups of our campaign as

\begin{equation}\label{eq:additonal_sales}
l_i=a_{T,i}-\frac{n_i}{k_i}a_{C,i},\qquad i=1,2
\end{equation}
 by simply scaling the responses of the control group to the
size of the target group. Note that this definition cannot be used without index, i.e. it cannot be used
 to define the uplift-based response of the overall campaign, since we do not assume equal target-control rates.
Unequal target control rates result in a unified control group being structurally different from the unified target group and, thus, the overall control group not
representative of the overall target group. For now, we define the uplift-based responses as the sum of the individual
uplift-based responses, i.e., $l:= l_1+l_2$. In tabular form this reads

 \begin{equation*} \label{tab:additional_sales}\begin{array}{l|l|l|l}
 \mbox{uplift} \quad &\quad \mbox{additional responses}\quad &\quad \mbox{no additional responses}\quad &\quad \mbox{total}\quad\\
\hline
\mbox{subgroup~1}\quad &\quad l_{1}\quad &\quad n_1-l_{1}\quad &\quad n_1\quad\\
\hline
\mbox{subgroup~2}\quad &\quad l_{2}\quad &\quad n_2-l_{2}\quad &\quad n_2\quad\\
\hline
\mbox{total}\quad &\quad l\quad &\quad n-l\quad &\quad n\quad\\
 \end{array}
 \end{equation*}

We next make the following model assumption: $a_{G,SG}$, $G=T,C$, $SG=1,2$ follow a binomial distribution with a response probability $0 < p_{G,SG} <1$.
We also assume that the responses
in the four target and control subgroups are independent of each other.

In the following, we want to investigate the hypothesis
\begin{equation}\label{eq:null_hypothesis}
p_{T,1}-p_{C,1} = p_{T,2}-p_{C,2}
\end{equation}
i.e. that the uplift is the same in both subgroups.

For that, we compute the expectation and the variance of the uplift-based responses by the known expectation and variance formula for a binomial distributed random variable, see Section~3.2
of~\cite{schinazi:2012}
and the standard calculation rules for expectation and variance, see Sections~2.3 and~2.4 of~\cite{schinazi:2012}

\begin{equation}\label{eq:expectation_li}
\E(l_i)=n_ip_{T,i}-n_ip_{C,i}=n_i(p_{T,i}-p_{C,i}),\qquad i=1,2
\end{equation}
and
\begin{equation}\label{eq:variance_li}
\Var(l_i)=n_ip_{T,i}(1-p_{T,i})+\frac{n_i^2}{k_i} p_{C,i}(1-p_{C,i}),\qquad i=1,2
\end{equation}

We estimate $p_{T,i}$ by $\hat p_{T,i}:=\frac{a_{T,i}}{n_i}$ and $p_{C,i}$ by $\hat p_{C,i}:=\frac{a_{C,i}}{k_i}$, $i = 1,2$ which are unbiased estimators and
converge to the probabilties by the law of large numbers, see Section~4.1 of~\cite{schinazi:2012}.

Next, we estimate the response propabilities of the unified target and control groups.
By elementary stochastic considerations we find the probability $p_T$ of the unified target group to be

\begin{eqnarray}\label{eq:p_T_definition}
p_{T} &=& \frac{n_1}n p_{T,1}+\frac{n_2}n p_{T,2}
 \end{eqnarray}
and it can be estimated by
\begin{eqnarray}\label{eq:hat_p_T_definition}
\hat p_{T} &:=& \frac{a_T}n = \frac{n_1}n\hat p_{T,1}+\frac{n_2}n\hat p_{T,2}
 \end{eqnarray}

For the control group, we could make an analogous defintion which however leads to biased results in Scenario~1.
An analogous definition for the control group would however lead to biased results considering the setup of Scenario~1 where we cannot make the
assumption of equal target-control group rates. We therefore define

\begin{eqnarray}
\label{eq:p_C_definition}
p_{C} &:=& \frac{n_1}n p_{C,1}+\frac{n_2}n p_{C,2},
 \end{eqnarray}
where the ''natural'' weighting factors $k$, $k_1$ and $k_2$ have been replaced by $n$, $n_1$ and $n_2$.
An unbiased estimator is

\begin{eqnarray}
\label{eq:hat_p_C_definition}
\hat p_{C} &:=& \frac{n_1}n\hat p_{C,1}+\frac{n_2}n\hat p_{C,2}
 \end{eqnarray}

 This definition of $p_C$ by the target group weighted control group response probabilities is crucial, as we will explain later.
 A discussion of the case $\hat p_{C}:=\frac{a_{C}}{k}$ will be given in Remark~\ref{rem:newer_statistic}.

We put

\begin{equation}\label{eq:est_expectation_li}
\hat e_i=n_i\left(\hat p_{T}-\hat p_{C}\right)
\end{equation}
as an estimator of the expectation of the additional responses in each individual campaign which is only valid under the null hypothesis~(\ref{eq:null_hypothesis}) (note the missing indices on the
right hand side).
Additionally, we define

\begin{equation}\label{eq:est_variance_li}
\hat v_i=n_i\hat p_{T,i}(1-\hat p_{T,i})+\frac{n_i^2}{k_i}\hat p_{C,i}(1-\hat p_{C,i})
\end{equation}
as an unbiased estimator of the variances of $l_i$.

In order to test the null hypothesis~(\ref{eq:null_hypothesis}), we define the net $\chi^2$-statistic

\begin{equation}\label{eq:net_chi_square}
\net:=\left(\frac{(l_1-\hat e_1)^2}{\hat v_1}+\frac{(l_2-\hat e_2)^2}{\hat v_2}\right)\cdot\frac{1}{\hat w_n\hat f_n}
\end{equation}
with the norming terms

\begin{eqnarray}
\nonumber \hat w_n &=& \frac{n_2}n\left[\hat p_{T,1}(1-\hat p_{T,1})+\frac{n_1}{k_1}\hat p_{C,1}(1-\hat p_{C,1})\right]\\
 \label{eq:hatw}                 && +  \frac{n_1}n\left[\hat p_{T,2}(1-\hat p_{T,2})+\frac{n_2}{k_2}\hat p_{C,2}(1-\hat p_{C,2})\right]
\\
 \hat f_n &=&\frac{\frac{n_2}{n}}{\hat p_{T,1}(1-\hat p_{T,1})+\frac{n_1}{k_1}\hat p_{C,1}(1-\hat p_{C,1})}
\label{eq:hatf}          +                  \frac{\frac{n_1}{n}}{\hat p_{T,2}(1-\hat p_{T,2})+\frac{n_2}{k_2}\hat p_{C,2}(1-\hat p_{C,2})}
\end{eqnarray}

Since $\hat e_i$ estimate the expectation of $l_i$ within the unified subgroups (and hence their sum the additional responses for the the overall group),
 it is easy heuristically to see with the law of large numbers that $\net$ will be close to~$0$ if both subgroups have the same uplift:

\begin{eqnarray*}
l_1-\hat e_1 &=& a_{T,1}-\frac{n_1}{k_1}a_{C,1}-n_1\left(\hat p_T-\hat p_C\right)\\
             &\approx &n_1 p_{T,1}-{n_1}p_{C,1}-n_1\left( \frac{n_1}n p_{T,1}+\frac{n_2}n p_{T,2}-\frac{n_1}n p_{C,1}-\frac{n_2}n p_{C,2}\right)\\
             &=&\left(n_1-\frac{n_1^2}{n}\right) p_{T,1}-\left(n_1-\frac{n_1^2}{n}\right)p_{C,1}-\frac{n_1n_2}np_{T,2}+\frac{n_1n_2}np_{C,2}\\
             &=&\frac{n_1n_2}n\left[\left(p_{T,1}-p_{C,1}\right)-\left(p_{T,2}-p_{C,2}\right)\right]\stackrel{(\ref{eq:null_hypothesis})}{=}0
\end{eqnarray*}
Analogously, $l_2-\hat e_2\approx 0$ under the null hypothesis. The terms $\hat v_i$, $\hat w_n$ and $\hat f_n$ ensure the $\chi^2$-distribution
of $\net$ when the null hypothesis holds.

Note from the last line that $\net$ scales with a factor of $\frac{n_1n_2}n$ when the uplifts in the subgroups are different.
In order to be able to compute $p$-values, we now want to prove that under the null hypothesis~(\ref{eq:null_hypothesis}), $\net$ follows asymptotically a $\chi^2$-distribution with one degree
of freedom. The proof follows the principles of~\cite{cramer:1945}, pp.~446, for the special case of a $2\times 2$ contingency table.

In order to show this convergence, we have to introduce some regularity conditions. Let $n_1$, $n_2$, $k_1$, $k_2$ and $k$ depend on $n$. Also suppose that
\begin{eqnarray}
\label{eq:regularity_condition n1n}\lim_{n\to\infty} \frac {n_1}n = s > 0,\quad \lim_{n\to\infty} \frac {n_1}{k_1} = t_1 > 0,\quad \lim_{n\to\infty} \frac {n_2}{k_2} = t_2 > 0
\end{eqnarray}
which ensure that the group sizes increase in a ''regular'' manner, which is a common assumption in such cases. If $t_1=t_2$,
representativity of the unified control group to the target group in the limit is ensured. This is the setup for using $\net$ in
Scenario~2 of Section~\ref{sec:test scenarios} (splitting one large group into subgroups).

Assumptions~(\ref{eq:regularity_condition n1n}) imply the following convergence
\begin{eqnarray}
\label{eq:regularity_condition n2n}\lim_{n\to\infty} \frac {n_2}n &=& \lim_{n\to\infty} \frac {n-n_1}n = 1-\lim_{n\to\infty} \frac {n_1}n = 1-s
\end{eqnarray}

Regarding the variances, we find the following implications
\begin{eqnarray}
\nonumber \lim_{n\to\infty}\frac{\Var(l_i)}{\hat v_i} &=& \lim_{n\to\infty}\frac{n_ip_{T,i}(1-p_{T,i})+\frac{n_i^2}{k_i} p_{C,i}(1-p_{C,i})}{n_i\hat p_{T,i}(1-\hat p_{T,i})+\frac{n_i^2}{k_i}\hat p_{C,i}(1-\hat p_{C,i})}\\
\label{eq:regularity_condition variance}  &=& \lim_{n\to\infty}\frac{p_{T,i}(1-p_{T,i})+\frac{n_i}{k_i} p_{C,i}(1-p_{C,i})}{\hat p_{T,i}(1-\hat p_{T,i})+\frac{n_i}{k_i}\hat p_{C,i}(1-\hat p_{C,i})}=1
\end{eqnarray}
by the above convergence of the estimators
\begin{eqnarray}
\label{eq:ptop}\lim_{n\to\infty}\hat p_{G,SG} = p_{G,SG},\qquad G=T,C,\qquad SG=1,2
\end{eqnarray}
 and regularity conditions~(\ref{eq:regularity_condition n1n}).

Remark that with analogous arguments as in~(\ref{eq:regularity_condition variance}) one gets
\begin{eqnarray}
\label{eq:regularity_condition variance_rez} \lim_{n\to\infty}\frac{\frac 1{\hat v_1}+\frac 1{\hat v_2}}{\frac 1{\Var(l_1)} +\frac 1{\Var(l_2)}}  =1
\end{eqnarray}
which we need later. Finally, we get
\begin{eqnarray}
\nonumber \lim_{n\to\infty}\hat w_n &=& (1-s)\left[p_{T,1}(1-p_{T,1})+t_1 p_{C,1}(1- p_{C,1})\right]\\
\label{eq:convergence_hatw}                       && +  s\left[ p_{T,2}(1-p_{T,2})+t_2 p_{C,2}(1- p_{C,2})\right] =: w
\\
\nonumber \lim_{n\to\infty}\hat f_n &=&
                            \frac{1-s}{p_{T,1}(1-p_{T,1})+t_1p_{C,1}(1-p_{C,1})}\\
\label{eq:convergence_hatf}                            && +\frac{s}{p_{T,2}(1-p_{T,2})+t_2p_{C,2}(1-p_{C,2})} =: f
\end{eqnarray}

We can now state our central theorem and show its proof.

\begin{theorem}\label{theo:net_chi_square}
Under the null hypothesis~(\ref{eq:null_hypothesis}), the regularity conditions~(\ref{eq:regularity_condition n1n})
and $p_{T,1},p_{T,2},p_{C,1},p_{C,2}\notin \{0,1\}$,
$\net$ follows for $n\to\infty$ a $\chi^2$-distribution with one degree of freedom.
\end{theorem}
\begin{proof}

Using the definitions of $l_1$, $l_2$, $\hat e_1$ and $\hat e_2$ one can see by elementary calculations that
\begin{eqnarray}
\label{eq:l2minuse2}l_2-\hat e_2 =  \frac{n_1}na_{T,2} - \frac{n_1n_2}{nk_2}a_{C,2} -\frac{n_2}{n}a_{T,1} +\frac{n_1n_2}{nk_1}a_{C,1}
=-\left(l_1-\hat e_1\right)
\end{eqnarray}
holds. Thus, putting this result into~(\ref{eq:net_chi_square}) we have (remark the index $1$ in the second term)
\begin{eqnarray}\label{eq:proof_first_step}
\hat f_n\hat w_n\net = \frac{(l_1-\hat e_1)^2}{\hat v_1} + \frac{(l_1-\hat e_1)^2}{\hat v_2}
\end{eqnarray}

In a first step we will show
\begin{eqnarray}\label{eq:proof_second_step}
\sqrt{\frac{n}{n_1n_2}}(l_1-\hat e_1)
\end{eqnarray}
is asymptotically normally distributed with a mean of $0$ and variance $w$ from~(\ref{eq:convergence_hatw}).
In the second step we will put the pieces together and prove that $\net$ follows a $\chi^2$-distribution with one degree of freedom.

In order to see that~(\ref{eq:proof_second_step}) is normally distributed, we add zeros to~(\ref{eq:l2minuse2}) and find the form
\begin{eqnarray}
\nonumber l_1-\hat e_1
 &=&\frac{n_2}n\left(a_{T,1}-n_1p_{T,1}\right)-\frac{n_1n_2}{nk_1}\left(a_{C,1}-k_1p_{C,1}\right)
 -\frac{n_1}n\left(a_{T,2}-n_2p_{T,2}\right) \\
\label{eq:l1minuse1} &&+\frac{n_1n_2}{nk_2}\left(a_{C,2}-k_2p_{C,2}\right)
 + \frac{n_1n_2}n\underbrace{\left[(p_{T,1}-p_{C,1})-(p_{T,2}-p_{C,2})\right]}_{\stackrel{(\ref{eq:null_hypothesis})}{=}0}
\end{eqnarray}
which simplifies due to the null hypothesis~(\ref{eq:null_hypothesis}) of equal uplifts in both subgroups.

Thus,
\begin{eqnarray}
\nonumber \sqrt{\frac n{n_1n_2}}\left(l_1-\hat e_1\right)
 &=&\underbrace{\frac{n_2}n\sqrt{\frac n{n_1n_2}}\sqrt{n_1p_{T,1}\left(1-p_{T,1}\right)}\underbrace{\frac{a_{T,1}-n_1p_{T,1}}{\sqrt{n_1p_{T,1}\left(1-p_{T,1}\right)}}}_{=: F_{T,1}}}_{=: H_{T,1}}\\
\nonumber &&-\underbrace{\frac{n_1n_2}{nk_1}\sqrt{\frac n{n_1n_2}}\sqrt{k_1p_{C,1}\left(1-p_{C,1}\right)}\underbrace{\frac{a_{C,1}-k_1p_{C,1}}{\sqrt{k_1p_{C,1}\left(1-p_{C,1}\right)}}}_{=: F_{C,1}}}_{=: H_{C,1}}\\
\nonumber &&-\underbrace{\frac{n_1}n\sqrt{\frac n{n_1n_2}}\sqrt{n_2p_{T,2}\left(1-p_{T,2}\right)}\underbrace{\frac{a_{T,2}-n_2p_{T,2}}{\sqrt{n_2p_{T,2}\left(1-p_{T,2}\right)}}}_{=: F_{T,2}}}_{=: H_{T,2}}\\
\nonumber &&+\underbrace{\frac{n_1n_2}{nk_2}\sqrt{\frac n{n_1n_2}}\sqrt{k_2p_{C,2}\left(1-p_{C,2}\right)}\underbrace{\frac{a_{C,2}-k_2p_{C,2}}{\sqrt{k_2p_{C,2}\left(1-p_{C,2}\right)}}}_{=: F_{C,2}}}_{=: H_{C,2}}
\end{eqnarray}
Since $a_{G,SG}$ follow a binomial distribution, the $F_{G,SG}$ converge to a standard normal distribution by the Central Limit theorem, see Section~4.2 of~\cite{schinazi:2012}.
With the convergencies
\begin{eqnarray*}
\frac{n_2}n\sqrt{\frac n{n_1n_2}}\sqrt{n_1} = \sqrt{\frac{n_2}{n}} &\to_{n_\to\infty} &\sqrt{1-s}\\
\frac{n_1n_2}{nk_1}\sqrt{\frac n{n_1n_2}}\sqrt{k_1} = \sqrt{\frac{n_1}{k_1}}\cdot \sqrt{\frac {n_2}n} &\to_{n_\to\infty} &\sqrt{t_1(1-s)}\\
\frac{n_1}n\sqrt{\frac n{n_1n_2}}\sqrt{n_2} = \sqrt{\frac{n_1}{n}} &\to_{n_\to\infty} &\sqrt{s}\\
\frac{n_1n_2}{nk_2}\sqrt{\frac n{n_1n_2}}\sqrt{k_2} = \sqrt{\frac{n_2}{k_2}}\sqrt{\frac{n_1}{n}} &\to_{n_\to\infty} &=\sqrt{t_2s}
\end{eqnarray*}
which follow from regularity conditions (\ref{eq:regularity_condition n1n}) and (\ref{eq:regularity_condition n2n})
we see that
\begin{eqnarray*}
H_{T,1}&\to_{n\to\infty}& N\left(0,(1-s)p_{T,1}\left(1-p_{T,1}\right)\right)\\
H_{C,1}&\to_{n\to\infty}& N\left(0,t_1(1-s)p_{C,1}\left(1-p_{C,1}\right)\right)\\
H_{T,2}&\to_{n\to\infty}& N\left(0,sp_{T,2}\left(1-p_{T,2}\right)\right)\\
H_{C,2}&\to_{n\to\infty}& N\left(0,t_2sp_{C,2}\left(1-p_{C,2}\right)\right)\\
\end{eqnarray*}
where $N(\mu,\sigma^2)$ denotes the normal distribution with mean $\mu$ and variance $\sigma^2$.
Since
\begin{eqnarray*}
(1-s)\left[p_{T,1}\left(1-p_{T,1}\right) +   t_1p_{C,1}\left(1-p_{C,1}\right)\right] +
 s\left[p_{T,2}\left(1-p_{T,2}\right) + t_2p_{C,2}\left(1-p_{C,2}\right)\right]=w,
\end{eqnarray*}
the independence of the $H_{G,SG}$ (which follows from the independence of $a_{G,SG}$) and the convolution theorem for the normal distribution (i.e. the sum of two independent normal random variables is also normal), we conclude that
$\sqrt{\frac n{n_1n_2}}\left(l_1-\hat e_1\right)$ converges under the null hypothesis~(\ref{eq:null_hypothesis}) to a normal distribution with a mean of $0$ and variance $w$,
which finalizes the first step of our proof.

We begin the second step by computing the following expression under the definition of the variances~(\ref{eq:variance_li}) and~(\ref{eq:hatf})
\begin{eqnarray}
\nonumber \lefteqn{\frac{n_1n_2}{n}\cdot\left(\frac 1{\Var(l_1)}+\frac 1{\Var(l_2)}\right)}\\
\nonumber &=&
\frac {\frac{n_2}n}{p_{T,1}\left(1-p_{T,1}\right)+\frac{n_1}{k_1}p_{C,1}\left(1-p_{C,1}\right)}+
\frac {\frac{n_1}n}{p_{T,2}\left(1-p_{T,2}\right)+\frac{n_2}{k_2}p_{C,2}\left(1-p_{C,2}\right)}
\\
\nonumber &\to_{n\to\infty}&
\frac {1-s}{p_{T,1}\left(1-p_{T,1}\right)+t_1p_{C,1}\left(1-p_{C,1}\right)}+
\frac {s}{p_{T,2}\left(1-p_{T,2}\right)+t_2p_{C,2}\left(1-p_{C,2}\right)}
\label{eq:convergence_reciporcals} =f
\end{eqnarray}
The convergence follows again by the regularity conditions (\ref{eq:regularity_condition n1n}) and (\ref{eq:regularity_condition n2n}).

Putting all the pieces together, we get the following representation of $\net$ from~(\ref{eq:proof_first_step}):

\begin{eqnarray*}
\net &=& \left((l_1-\hat e_1)^2\left(\frac 1{\hat v_1}+\frac 1{\hat v_2}\right)\right)\frac{1}{\hat w_n\hat f_n}\\
     &=& \left(\underbrace{\underbrace{\sqrt{\frac n{n_1n_2}}(l_1-\hat e_1)}_{\to N(0,w)}\cdot\frac 1{\sqrt w}}_{\to N(0,1)}\right)^2
       \cdot\underbrace{\frac{\frac{n_1n_2}n\cdot\left(\frac 1{\Var(l_1)}+\frac 1{\Var(l_2)}\right)}{\hat f_n}}_{\to f/f=1}
     \cdot\underbrace{\frac{w}{\hat w_n}}_{\to 1}
        \cdot\underbrace{\frac{\frac 1{\hat v_1}+\frac 1{\hat v_2}}{\frac 1{\Var(l_1)}+\frac 1{\Var(l_2)}}}_{\to 1} \\
\end{eqnarray*}

The convergencies follow by~(\ref{eq:regularity_condition variance_rez}) to~(\ref{eq:convergence_hatf}) and~(\ref{eq:proof_second_step}).
Thus, $\net$ is asymptotically $\chi^2$-distributed with one degree of freedom which is the distribution of the square of a standard normal distributed random variable.

\end{proof}

\begin{remark}\label{rem:newer_statistic}
In the case of Scenario~2 in Section~\ref{sec:test scenarios} (the unified target and control groups are representative due to equal target-control rates of the
subgroups) the ''natural'' defintion of
$\hat p_C=\frac{a_C}{k}$ can be used to define the $\net$-statistic. It then also follows asymptotically a $\chi^2$-distribution with one degree of freedom,
however, the regularity conditions~(\ref{eq:regularity_condition n1n})  have to be expanded by the assumption that the
convergencies are superlinear, i.e. $\lim_{n\to\infty}n\left(\frac{n_1}n-s\right)=0$ and analogously for $\frac{n_1}{k_1}$ and $\frac{n_2}{k_2}$. Also, the latter
terms converge to the same number $t_1=t_2=:t$. The proof itself becomes more complicated since terms which cancel in the proof of Theorem~\ref{theo:net_chi_square} in equations like~(\ref{eq:l2minuse2}) and
(\ref{eq:l1minuse1}) only vanish in the limit. In the following we note this version as $\netnewer$.
\end{remark}

\begin{remark}\label{rem:old_statistic}
In~\cite{michel:schnakenburg:martens:2013} another slightly different version of $\net$ for Scenario~2, based on $\netnewer$ from Remark~\ref{rem:newer_statistic} is presented. It differs since the norming
term $\frac 1{\hat w_n\hat f_n}$ is omitted and the denominators are defined by
\begin{equation}\label{eq:est_variance_li_variant}
\hat v_i=n_i\hat p_{T}(1-\hat p_{T})+\frac{n_i^2}{k_i}\hat p_{C}(1-\hat p_{C})
\end{equation}
with $\hat p_C = \frac{a_C}k$. In comparison to (\ref{eq:est_variance_li}), the estimation of the variance is based on the whole sample ($p_{T}$ and $p_C$) and not only the subgroup specific parts $p_{T,i}$ and $p_{C,i}$. With the above arguments, it can be seen that under the
 regularity conditions mentioned in Remark~\ref{rem:newer_statistic}, this statistic also follows a $\chi^2$-distribution with one degree of freedom. However, the null hypothesis has to be expanded in order to include
$p_{C,1}=p_{C,2}$, i.e. the hypothesis of equal random noise. In the following we note this version as $\netold$.
\end{remark}

\section{Alternative approaches}\label{sec:alternative}

In this section we present two alternative methods for approaching the mentioned test scenarios.
The first one is the use of contrasts in a variance analytical context shown in~\cite{sasinstitute:2012},
the second is a test statistic from net scoring or uplift modelling as presented in~\cite{radcliffe:surry:2011}.

If for an individual observation response is coded as $0$ (= no response) and $1$ (= response), we can use variance analysis in order to decide our testing problem of equal uplifts.
We regard the group (target or control, subscript: G) as one factor and the two different campaings (or subgroups of one campaign, subscript: SG) as another factor. The empirical means of the four groups are just the estimated probabilities
$\hat p_{G,SG}$ for a response within the four groups. A classical variance analysis would now compare them in order to search for differences. However, we are interested
in the linear hypothesis of the theoretical group means $p_{T,1}-p_{C,1}=p_{T,2}-p_{C,2}$. Such linear hypothosis can be investigated by means of contrasts with the computation of appropriate
statistics and $p$-values, see Section~3.2 of~\cite{sasinstitute:2012}. However, the standard assumptions for variance analysis, namely independence of the observations,
normally distributed observations and equal variances in the groups (homoscedasticity), need to be fullfilled.

There are two reasons why we prefer our $\net$-method.
Firstly, ''response'' and ''no response'' refer to
a binary target variable and by aggregating over the observations we are in the area of count data, for which the $\chi^2$-family of statistics are
 especially constructed. Variance analysis is primarily aimed at
continuous target variables but under certain conditions can be applied to count data.

The second reason emerges from the assumptions of variance analysis. The assumption of independence is standard and is also required for $\net$. The assumption
of normal data can be relaxed when the samples sizes are large, which is usually the case in our marketing applications. However, the third assumption of homoscedasticity is critical in our view.
Since the number of responses in each group is binomially distributed, it can be approximated by a normal distribution by the law of large numbers.
Different probability parameters for the binomial distribution will automatically result in different variances, therefore spoiling the homoscedasticity condition,
 compare~(\ref{eq:variance_li}). Since even under the null hypothesis~(\ref{eq:null_hypothesis})
$p_{C,1}$ and $p_{C,2}$ are explicitly allowed to differ from each other, we have heteroscedasticity. Variance analysis is robust to heteroscedasticity when
sample sizes are equal, see Section~3.5 of~\cite{sasinstitute:2012}. However, in marketing applications the control group is usually much smaller than the target group
(e.g. $10\%$). Also, Section~3.5 of~\cite{sasinstitute:2012} shows that heteroscedasticity with unequal sample sizes can lead to
an increased type~I and type~II error rate. Thus, we prefer our method which does not suffer from such defects.

The second alternative method is also inspired by a general linear model in combination with a regression. It is described in detail in Section~6.2 of~\cite{radcliffe:surry:2011}.
We will give the formulae here with our notations. Define norming terms by

\begin{eqnarray*}
C_{44} &:=& \frac 1{n_1} + \frac 1{n_2} + \frac 1{k_1} + \frac 1{k_2}\\
SSE &:=& n_1\hat p_{T,1}(1-\hat p_{T,1}) +  n_2\hat p_{T,2}(1-\hat p_{T,2}) +  k_1\hat p_{C,1}(1-\hat p_{C,1}) +  k_2\hat p_{C,2}(1-\hat p_{C,2})
\end{eqnarray*}
and the statistic by
\begin{eqnarray}
\tnet &:=& \frac{(n+k-4)(\hat p_{T,1}-\hat p_{C,1}-(\hat p_{T,2}-\hat p_{C,2}))^2}{C_{44}\cdot SSE}
\end{eqnarray}

It has the notation $\tnet$ since it is implied (although neither explicitly stated nor proved) in~\cite{radcliffe:surry:2011} that
$t_{{\rm net}}$ follows asymptotically a $t$-distribution with $n+k-4$ degrees of freedom. Since $n+k-4$ is quite large in our applications, the $t$-distribution
can be approximated by the standard normal distribution. Thus, $\tnet$ also follows asymptotically a $\chi^2$-distribution with one degree of freedom. In the simulations
that follow in the next section, we will see that this statement seems to be true when testing for the right null hypothesis.

\section{Simulation study}\label{sec:simulation}

In this section, we use a simulation study in order to compare the five approaches presented above:

\begin{itemize}
\item[1.] $\net$ with $p$-value $p_{\net}$
\item[2.] $\netnewer$ with modifications from Remark~\ref{rem:newer_statistic} with $p$-value $p_{\netnewer}$
\item[3.] $\netold$ with modifications from Remark~\ref{rem:old_statistic} with $p$-value $p_{\netold}$
\item[4.] the contrast approach with $p$-value $p_{\rm con}$
\item[5.] $\tnet$ with $p$-value $p_{\tnet}$
\end{itemize}

For each of the following simulations we consider a fixed set of numbers $n_1$, $n_2$, $k_1$, $k_2$ and probabilities $p_{T,1}$, $p_{T,2}$, $p_{C,1}$, $p_{C,2}$.
The $a_{T,1}$, $a_{T,2}$, $a_{C,1}$ and $a_{C,2}$ are binomial $B(n_1,p_{T,1})$-, $B(n_2,p_{T,2})$-, $B(k_1,p_{C,1})$- and $B(k_2,p_{C,2})$-distributed random variables.
With these our basic stochastic model as shown in Section~\ref{sec:modified chi square} is completely described.
Being binomially distributed, $a_{G,SG}$ are easily simulated with any standard statistical software package. Since
all five statistics and $p$-values above are functions of the $n_i$, $k_i$ and $a_{G,SG}$, they and their
corresponding $p$-values can be computed. For a fixed set of parameters, we repeat this $b$ times, usually choosing $b=100$. We then sort the corresponding
data set by the $p$-values $p_{\net}$ and denote the resulting numbers by $p_{\net}^{i:b}$, $p_{\netnewer}^{i:b}$, $p_{\netold}^{i:b}$, $p_{\rm con}^{i:b}$, $p_{\tnet}^{i:b}$, $i=1,\ldots, b$.
Remark that only the $p_{\net}^{i:b}$ are necessarily in ascending order for $i=1,\ldots, b$.
We then plot each of the five $p$-value series against the set of $\frac ib$, $i=1,\ldots, b$, i.e. the points $\left(\frac ib, p_x^{i:b}\right)$ for
$x=\net,\netnewer,\netold,{\rm con},\tnet$. If the null hypothesis behind each test is fulfilled,
the $p$-values follow a uniform distribution on $(0,1)$ and the plotted points scatter around the diagonal in this probability plot. However, if the null hypothesis is not fulfilled,
we will find deviations from the diagonal.

Table~\ref{tab:simulationsstruktur} shows the parameter values behind the results in Figures~\ref{fig:bild1} to~\ref{fig:bild7}.

\begin{table}[htbp]
\centering
\begin{tabular}{lccccccccc}\toprule
                          & aimed at          & $n_1$     & $n_2$        & $k_1$         & $k_2$   & $p_{T,1}$  & $p_{T,2}$   & $p_{C,1}$   & $p_{C,2}$ \\
                          & error             &           &              &               &         &            &             &             &            \\\colrule
  Figure~\ref{fig:bild1}  & type I            & $50,000$  & $50,000$     & $5,000$       & $5,000$ & $10\%$     & $10\%$      & $9\%$       & $9\%$      \\
  Figure~\ref{fig:bild2}  & type I            & $50,000$  & $50,000$     & $5,000$       & $5,000$ & $5\%$      & $51\%$      & $4\%$       & $50\%$      \\
  Figure~\ref{fig:bild3}  & type I            & $100,000$ & $20,000$     & $10,000$      & $2,000$ & $5\%$      & $51\%$      & $4\%$       & $50\%$      \\
  Figure~\ref{fig:bild4}  & type I            & $100,000$ & $20,000$     & $10,000$      & $2,000$ & $51\%$     & $5\%$       & $50\%$      & $4\%$      \\
  Figure~\ref{fig:bild5}  & type II           & $50,000$  & $50,000$     & $5,000$       & $5,000$ & $11\%$     & $10\%$      & $9\%$       & $9\%$      \\
  Figure~\ref{fig:bild6}  & type II           & $50,000$  & $50,000$     & $5,000$       & $5,000$ & $6\%$      & $51\%$      & $4\%$       & $50\%$      \\
  Figure~\ref{fig:bild7}  & type II           & $50,000$  & $50,000$     & $5,000$       & $10,000$& $5\%$      & $52\%$      & $4\%$       & $50\%$      \\ \botrule
 \end{tabular}
\caption{Parameters used for the seven simulations.}
\label{tab:simulationsstruktur}
\end{table}

The control groups were chosen to be roughly $10\%$ of the target group which is a quite common target-control rate. Also, the absolute group sizes are not unusual
in practice.

In general, all figures show one fact: there are very small differences between $p_{\rm con}$ and $p_{\tnet}$. Thus, in essence, although not clear from the description
in Section~\ref{sec:alternative} they seem to be the same method. This means that all the criticism stated for $p_{\rm con}$ carries over to $p_{\tnet}$.
The points of
criticism are supported by the simulations.

In Figures~\ref{fig:bild1} to~\ref{fig:bild4} we investigate if the tests control the type~I error, since the uplift in both subgroups is $1\%$. We vary the
rate between the target groups and the levels of the random noise in the subgroups.
The results are as follows:
\begin{itemize}
\item When both target groups have roughly the same size and the random noise is equal, all methods deliver nearly the same results and maintain the type~I error (Figure~\ref{fig:bild1}).
\item When target group sizes are the same, however random noises are on a different scale, $\netold$ gives slightly different results, however, the type~I
      error is still kept by all methods (Figure~\ref{fig:bild2}).
\item When target group sizes differ and the smaller group has the larger random noise, only $\net$ and $\netnewer$ are able to keep the type~I error.
      $\tnet$ has the smallest $p$-values and is no reliable in this case (Figure~\ref{fig:bild3}).
\item When target groups differ heavily and the larger group has the larger random noise, only $\net$ and $\netnewer$ are able to control the type~I error.
      The other methods have $p$-values too high and are not reliable tests here (Figure~\ref{fig:bild4}).
\end{itemize}

\begin{figure}
\centering
\includegraphics[width=0.95\textwidth]{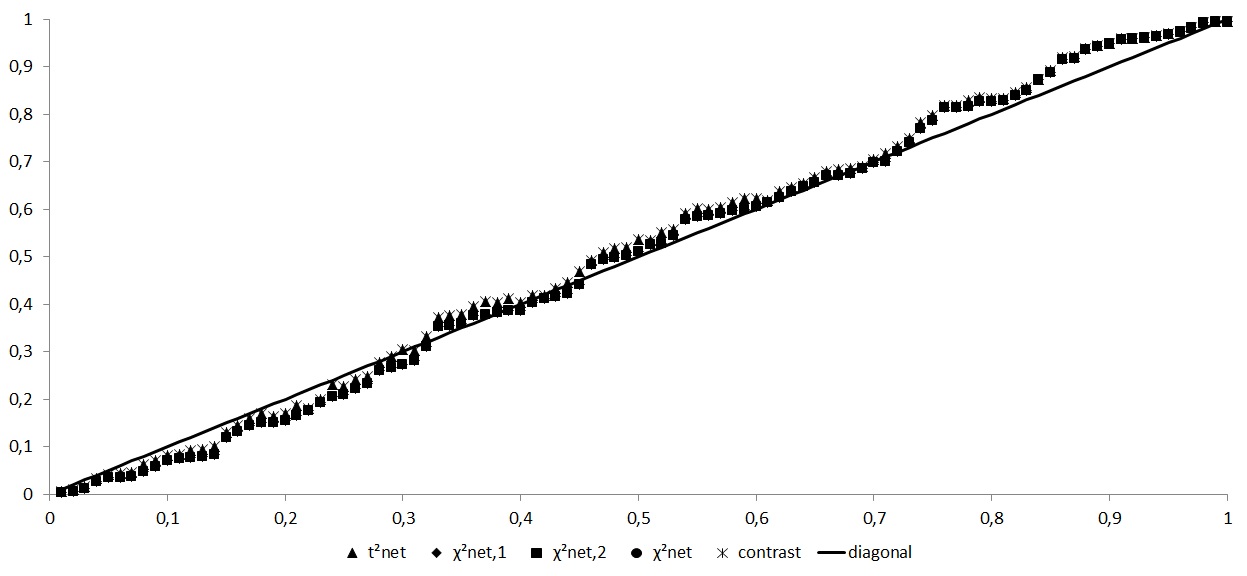}
\caption{Check for type~I error with equal group sizes and equal random noise.}
\label{fig:bild1}
\end{figure}

\begin{figure}
\centering
\includegraphics[width=0.95\textwidth]{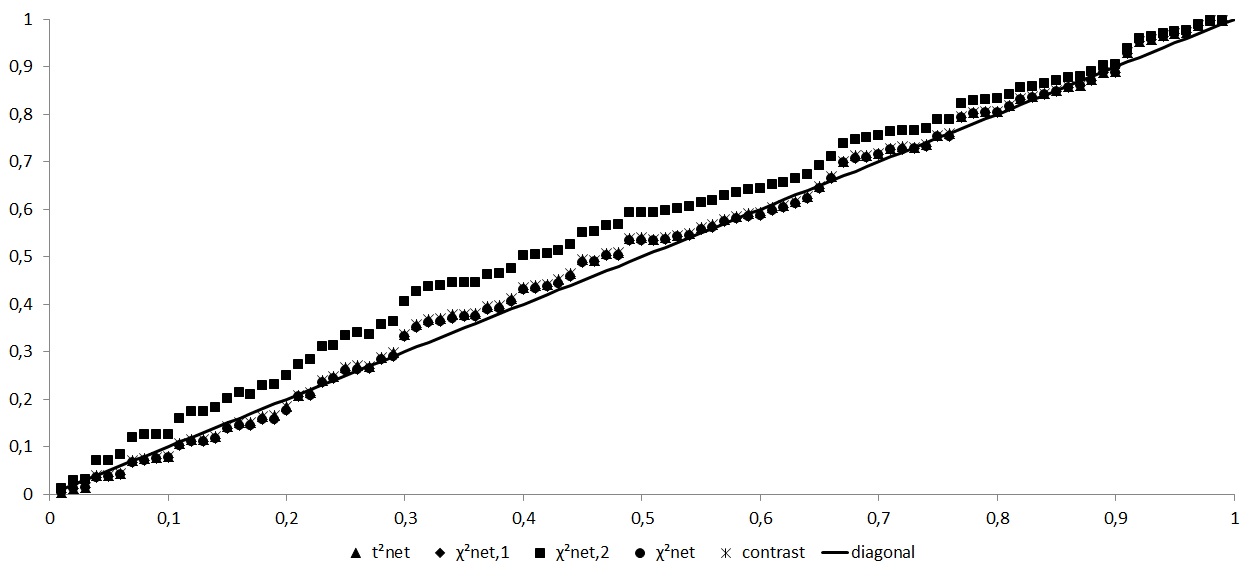}
\caption{Check for type~I error with equal group sizes and different random noise.}
\label{fig:bild2}
\end{figure}

\begin{figure}
\centering
\includegraphics[width=0.95\textwidth]{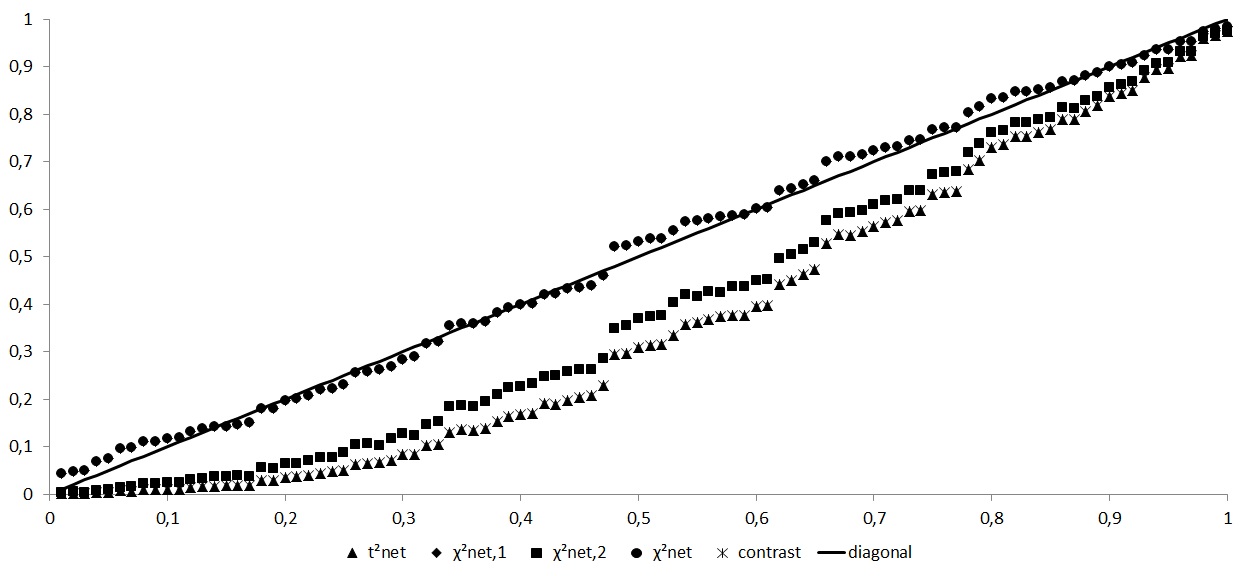}
\caption{Check for type~I error with different group sizes and different random noise (small group with large noise).}
\label{fig:bild3}
\end{figure}

\begin{figure}
\centering
\includegraphics[width=0.95\textwidth]{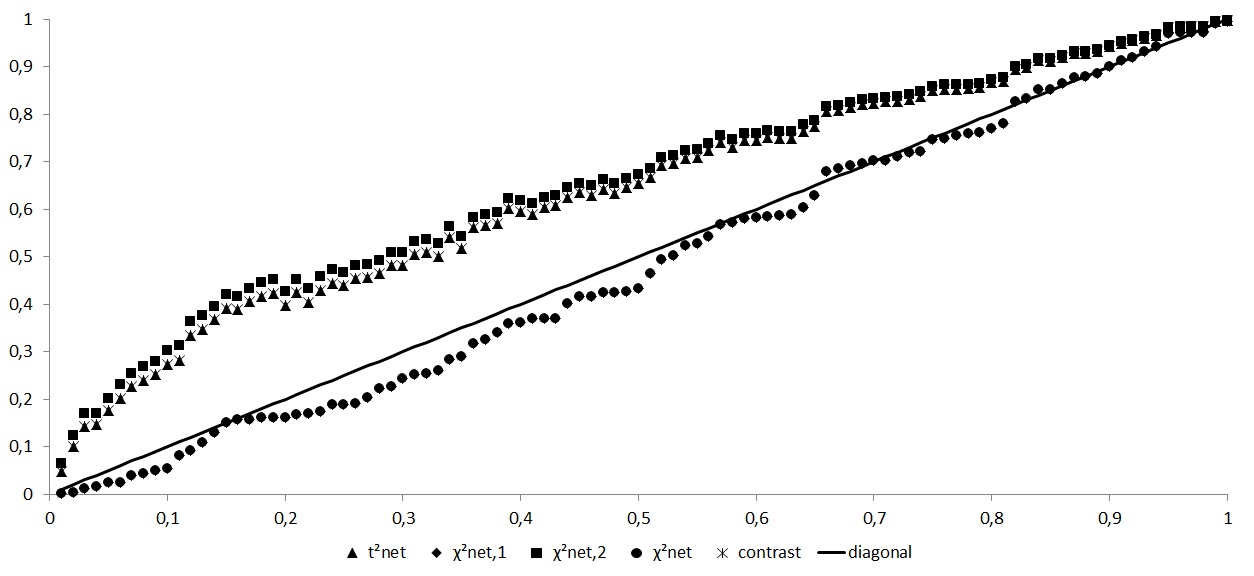}
\caption{Check for type~I error with different group sizes and different random noise (large group with large noise).}
\label{fig:bild4}
\end{figure}

In Figures~\ref{fig:bild5} to~\ref{fig:bild7}, we investigate the type~II error with an uplift of $2\%$ in one subgroup and of $1\%$ in the other subgroup.
Further, we only look at the case of equal target group sizes, since in the other case only $\net$ and $\netnewer$
can reasonably be used by the above results. Here we see:
\begin{itemize}
\item When random noise is roughly the same size, all methods detect deviations from the null hypothesis with $\net$ showing the smallest $p$-values (Figure~\ref{fig:bild5}).
\item When random noises are on a different scale, $\netold$ has a notably larger type~II error, the others are roughly equal (Figure~\ref{fig:bild6}).
\item When the target-control rate differs between the subgroups, $\netnewer$ and $\netold$ are useless and not shown since they were not constructed for this case. Of
the remaining methods, $\net$ is clearly the best one with the smallest type~II error (Figure~\ref{fig:bild7}).
\end{itemize}

\begin{figure}
\centering
\includegraphics[width=0.95\textwidth]{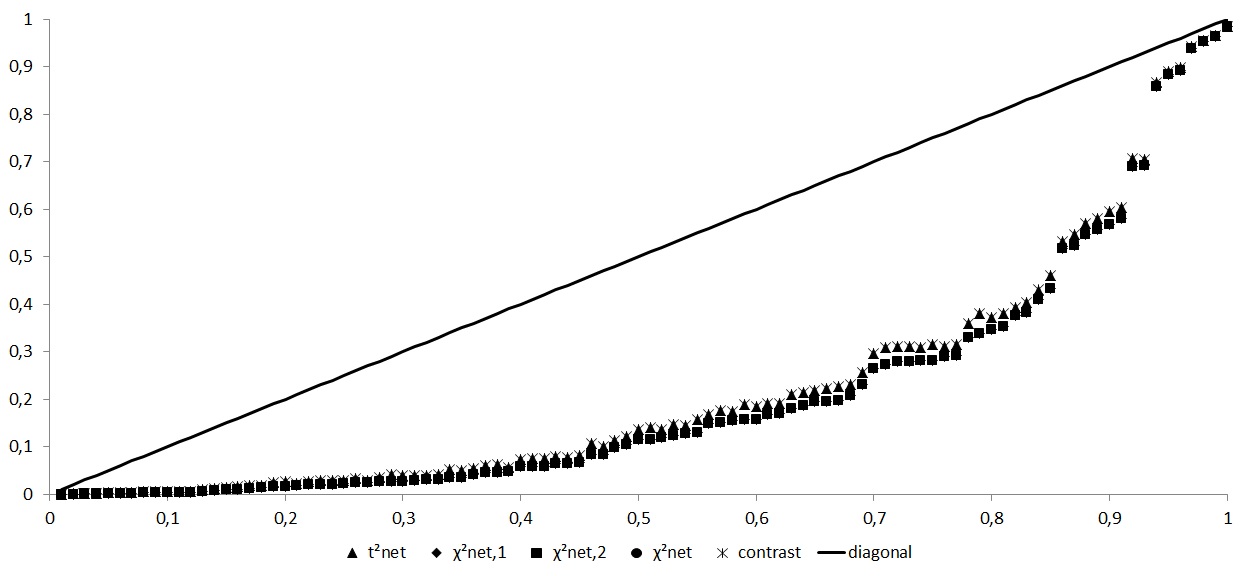}
\caption{Check for type~II error with equal group sizes and equal random noise.}
\label{fig:bild5}
\end{figure}

\begin{figure}
\centering
\includegraphics[width=0.95\textwidth]{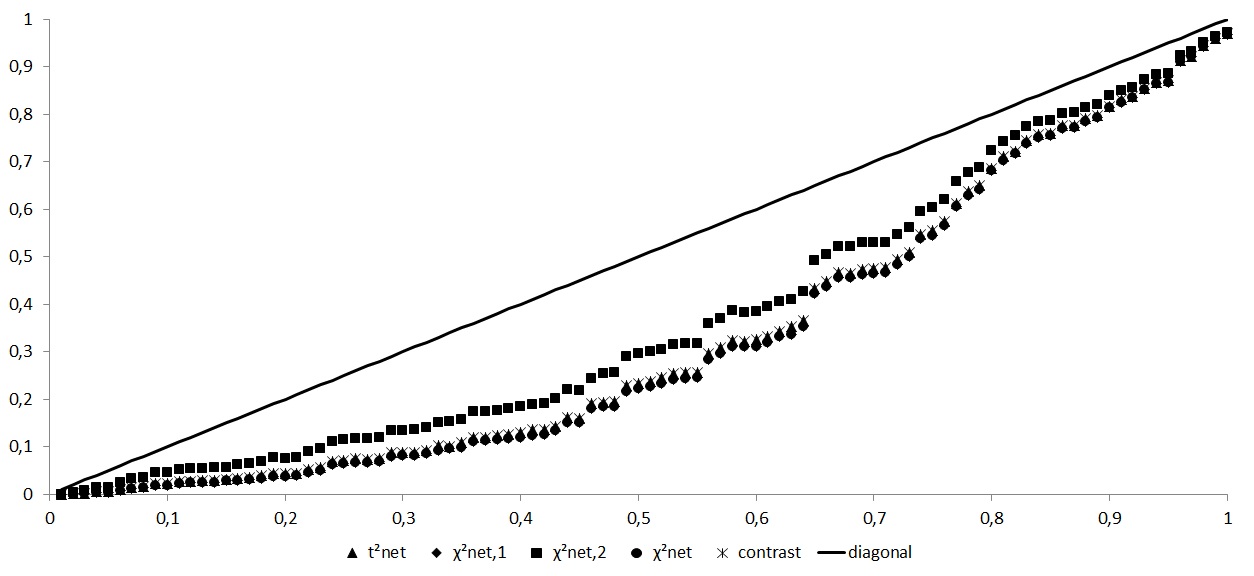}
\caption{Check for type~II error with equal group sizes and different random noise.}
\label{fig:bild6}
\end{figure}

\begin{figure}
\centering
\includegraphics[width=0.95\textwidth]{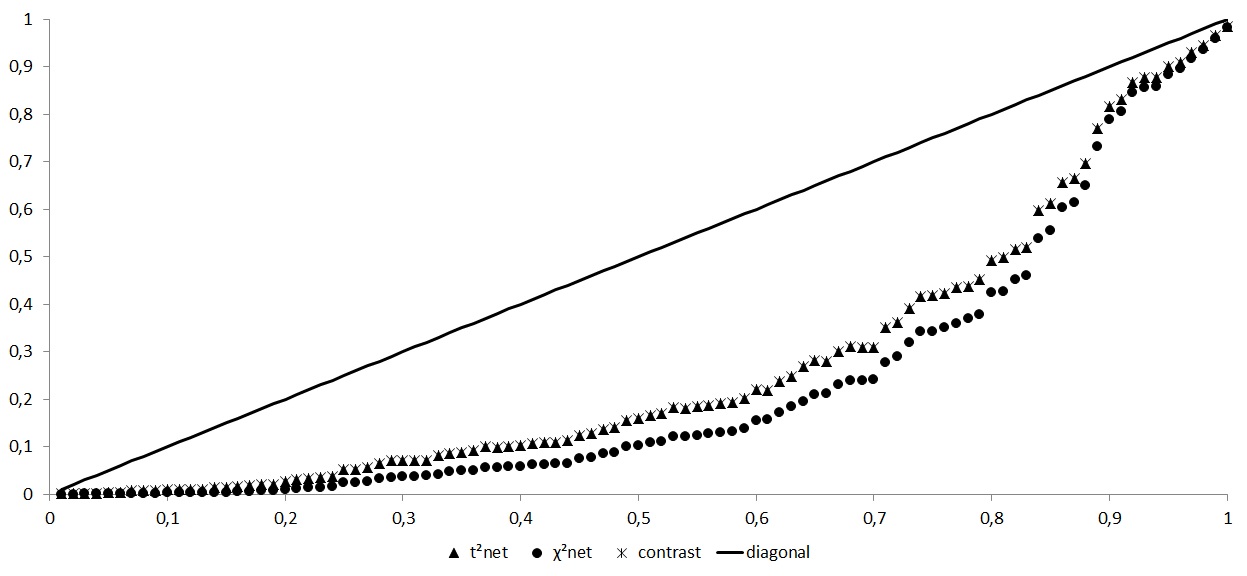}
\caption{Check for type~II error with equal group sizes and different random noise under different target-control rate.}
\label{fig:bild7}
\end{figure}

The results of the simulations are quite clear: $\net$ should be the method of choice when testing for equal uplifts. It seems that $\tnet$ and the contrast method
not only test for the null hypothesis of equal uplifts but additionally assume equal random noise, as does $\netold$, see Remark~\ref{rem:old_statistic}.
A mathematical investigation and comparison of the type~II error rates of all methods, but especially
$\net$ and $\tnet$, is still an open issue.

\section{Application to real data}\label{sec:application}

We next want to apply the statistics to real
world data, in particular covering both scenarios from Section~\ref{sec:test scenarios}, i.e. one campaign is split up or two different campaigns are joint
together.

A campaign was done by a financial institute with the aim of generating appointments with regard to provision for one's old age. For that purpose, $167$ thousand customers
were selected to receive a letter and a phone call inviting them to make an appointment with their bank advisor in order to discuss their hedging situation in old age.
A response was counted if an appointment was made. A control group of $13$ thousand customers was created.
Accordingly, appointments were counted as responses.

The first thing of interest for the bank was if gender is a decisive factor in response to the campaign. The bank had the hypothesis that women are more affine
to provision for old age since they tend to think more about saving for the future than men who are supposed to spend money more quickly.

Table~\ref{tab:gendercomparison} shows the results of the campaign, seperated by gender.

\begin{table}[htbp]
\centering
\begin{tabular}{lcccccccc}\toprule
  campaign~1  & \multicolumn{3}{c}{target}& & \multicolumn{3}{c}{control} & uplift\\    \colrule
              & persons & responses & rate && persons  & responses & rate &  \\\colrule
  women       & $81,770$  & $5,656$     & $6.92\%$       && $6,391$   & $373$        & $5.84\%$       & $1.08\%$ \\
  men         & $85,257$  & $6,231$      & $7.31\%$       && $6,699$    & $443$        & $6.61\%$       & $0.70\%$ \\ \botrule
 \end{tabular}
\caption{Comparison of campaign results by gender.}
\label{tab:gendercomparison}
\end{table}

The results seem to confirm the stronger affinity of women to provision, since women had an uplift of $1.08\%$ and men one of $0.70\%$.
A classical $\chi^2$-test for homogenity shows for men and women that the response rates in target and control group are significantly
different from each other in both groups, i.e. the uplifts are significantly different from $0$.
The question is,
however, if the uplifts are statistically different from each other. In order to answer this question, we compute $\net=0.7643$ and the corresponding $p$-value $p_{\net}=0.3820$. Thus, in this campaign,
the uplifts of women and men are not significantly different and the hypothesis above cannot be confirmed. Since we are in Scenario~2
($\frac {n_1}{k_1}=\frac{81,770}{6,391}\approx 12.79\approx 12.73\approx\frac{85,257}{6,699}=\frac {n_2}{k_2}$), the two alternatives from
Remarks~\ref{rem:newer_statistic} and~\ref{rem:old_statistic} can also be computed. They result in $\netnewer=0.7648$
and  $\netold=0.7622$  with corresponding $p$-values of $p_{\netnewer}=0.3818$ and $p_{\netold}=0.3827$.
Thus, the alternatives lead to almost exactly the same results. The approach by means of contrasts generates the $p$-value $p_{\rm con}=0.4075$, i.e. no significance. The statistic
$\tnet=0.6861$ with the same $p$-value of $p_{\tnet}=0.4075$ leads to the same conclusion.

Parallel to Campaign~1, another campaign was carried out which had the same objective (to get customers to make an appointment with their bank advisor to talk about
old age provision) and the same means (firstly a letter was sent and afterwards, some of the customers were called), however, the letter had a different layout and text.
Also, the customers were different since the second campaign
aimed at customers with higher income. Was the effort more successful for wealthier clients? Table~\ref{tab:campaigncomparison} shows the results.

\begin{table}[htbp]
\centering
\begin{tabular}{lcccccccc}\toprule
                 & \multicolumn{3}{c}{target}& & \multicolumn{3}{c}{control} & uplift\\    \colrule
                 & persons & responses & rate && persons  & responses & rate &  \\\colrule
  campaign~1     & $167,027$ & $11,887$     & $7.12\%$       && $13,090$   & $816$        & $6.23\%$       & $0.88\%$ \\
  campaign~2     & $44,356$  & $3,447$      & $7.77\%$       && $7,987$    & $492$        & $6.16\%$       & $1.61\%$ \\ \botrule
 \end{tabular}
\caption{Comparison of results for parallel campaigns aimed at different customer segments.}
\label{tab:campaigncomparison}
\end{table}

A classical $\chi^2$-test once again shows that the campaigns themselves were successful since both uplifts are significantly different from $0$.
When comparing the uplifts to each other, the uplift for the wealthier clients is $1.61\%$ which is above the uplift of $0.88\%$ for the middle-class clients.
In order to check for statistical significance, we compute $\net=3.8661$ and $p_{\net}=0.0493$ which is significant at the usual $5\%$-level. Thus,
there is evidence that a campaign with the above targets is more effective for customers with higher income.
Since $\frac {n_1}{k_1}=\frac{167,027}{13,090}\approx 12.76\gg 5.55\approx\frac{44,356}{7,987}=\frac {n_2}{k_2}$, we are in Scenario~1 and $\netnewer$
and $\netold$ cannot be applied.

However, the alternative approach by means of contrasts can be applied. The $p$-value here is $p_{\rm con}=0.0626$ and, thus, not significant, although close
to significant.
$\tnet$-statistic gives the value of $3.4672$ with a $p$-value of $p_{\tnet}=0.0626$, thus showing no significance in contrast to $\net$.
This is another indication that the statistical power of the $\tnet$ or the contrast method is lower than that of $\net$.
This example too shows that both methods ($\tnet$ and contrast) coincide at least quantitatively.

\section{Discussion and outlook}\label{sec:discussion}

In this article, we have presented a new statistic, based on the classical $\chi^2$ which is appropriate when statistically comparing the uplift of two campaigns.
We have proved that its asymptotic distribution is a $\chi^2$-distribution with one degree of freedom and shown its practical applicability by using it on real data to
decide a real life problem. We have also shown by means of simulation that it seems to be superior to the already existing alternative approaches by fixing the type~I error and showing
smaller type~II errors.

However, an open issue remains in the comparison of the method presented here with the alternative ones by mathematical means.

The statistic $\net$ has been presented here for the two sample (campaign) case. A generalization of the formula of $\net$ to the $j>2$ sample case seems straightforward by

\begin{equation*}\label{eq:multivariate_net_chi_square}
\net:=\frac{1}{\hat w_n\hat f_n}\sum_{i=1}^j\frac{(l_i-\hat e_i)^2}{\hat v_i}
\end{equation*}

However, the definition of suitable norming terms $\hat w_n$ and $\hat f_n$ is yet unclear. The conjecture is, of course, that
this generalized version is asymptotically distributed as the $\chi^2$ distribution with $j-1$ degrees of freedom. A mathematical assessment of that
assertion still needs to be done.

For the classical $\chi^2$-statistic, rules of thumb are known when the approximative distribution is valid (e.g. the expected frequency in each
cell must be larger than 5, i.e. $np_{i} >5$, see Section 4.2 of~\cite{falk:2003}).
Such rules are still missing in the $\net$-case.

Besides its applications to testing problems in marketing performance measurement, this statistic can also be used a scoring context, the so called net scoring.
Especially its application to the construction of decision trees was described in~\cite{michel:schnakenburg:martens:2013}.

However, marketing is not the only area in which our statistic can be applied. It is useful in all testing scenarios where the effect of a treatment
is investigated under the condition that the desired result could also appear by itself.
The medical application seems to suggest itself since typically the impact of different drugs is compared well-knowing that a certain percentage of patients
will recover even without treatment.
 Examples of the net effect in medicine are shown in~\cite{jaskowski:jaroszewicz:2012}
and~\cite{nassif:2013}. We hope this article stipulates research in this area and the discovery of many more areas of application.

\section*{Acknowledgements}

The authors would like to thank Hans Fischer, Michael Falk and Johannes Hain for their support on
this work as well as Altran GmbH \& Co. KG.

\nocite{hansotia:rukstales:2002}
\nocite{rzepakowski:jaroszewicz:2011}

\bibliographystyle{gSTA}
\bibliography{bibliography_net}

\begin{thebibliography}{10}
\providecommand{\url}[1]{\normalfont{#1}}
\providecommand{\urlprefix}{Available from: }
\providecommand{\eprint}[2][]{\url{#2}}

\bibitem{radcliffe:simpson:2008}
Radcliffe N, Simpson R. Identifying who can be saved and who will be driven
  away by retention activity. Journal of Telecommunications Management.
  2008;\hspace{0pt}1(2):168--176.

\bibitem{lo:2002}
Lo V. The true lift model - a novel data mining approach to response modeling
  in database marketing. SIGKDD Explorations. 2002;\hspace{0pt}4(2):78--86.

\bibitem{falk:2014}
Falk M, Fischer H, Hain J, Marohn F, Michel R. Statistik in theorie und praxis
  - mit anwendungen in r. Munich: Springer; 2014. to appear.

\bibitem{radcliffe:surry:2011}
Radcliffe N, Surry P. Real-world uplift modeling with significance-based uplift
  trees. 2011;\hspace{0pt}Technical Report, Stochastic Solutions.

\bibitem{michel:schnakenburg:martens:2013}
Michel R, Schnakenburg I, von Martens T. Methods of variable pre-selection for
  netscore modeling. Journal of Research in Interactive Marketing.
  2013;\hspace{0pt}7(4):257--268.

\bibitem{falk:2003}
Falk M, Marohn F, Tewes B. Foundations of statistical analyses - examples with
  sas. Basel: Birkh\"{a}user; 2003.

\bibitem{schinazi:2012}
Schinazi R. Probability with statistical applications. 2nd ed. Birkh\"{a}user;
  2012.

\bibitem{cramer:1945}
Cram\'er H. Mathematical methods of statistics. Princeton: University Press;
  1945.

\bibitem{sasinstitute:2012}
SAS. Statistics 2: Anova and regression course notes. Cary: SAS Institute Inc.;
  2012.

\bibitem{jaskowski:jaroszewicz:2012}
Jaskowski M, Jaroszewicz S. Uplift modeling for clinical trial data. ICML 2012
  Workshop on Clinical Data Analysis. 2012;\hspace{0pt}.

\bibitem{nassif:2013}
Nassif H, Kuusisto F, Burnside E, Page D, Shavlik J, Santos~Costa V. Score as
  you lift (sayl): A statistical relational learning approach to uplift
  modeling. Proceedings of the European Conference on Machine Learning (ECML).
  2013;\hspace{0pt}.

\bibitem{hansotia:rukstales:2002}
Hansotia B, Rukstales B. Incremental value modeling. Journal of Interactive
  Marketing. 2002;\hspace{0pt}16(3):35--46.

\bibitem{rzepakowski:jaroszewicz:2011}
Rzepakowski P, Jaroszewicz S. Decision trees for uplift modeling with single
  and multiple treatments. Knowledge and Information Systems.
  2011;\hspace{0pt}28(2):303--327.

\end{thebibliography}

\end{document}